\documentclass[12pt]{article}
\usepackage[english]{babel}
\usepackage[all,cmtip]{xy}
\usepackage{amssymb,amsmath,mathrsfs,amsthm}
\usepackage{color}
\newcommand{\rene}[1]{{\color{blue} \sf $\spadesuit\spadesuit\spadesuit$ Ren\'e: [#1]}}
\newcommand{\defi}[1]{\textsf{#1}} 

\newcommand\wt[1]{\widetilde{#1}}

\newtheorem{thm}{Theorem}

\newtheorem{defn}[thm]{Definition}
\newtheorem{cor}[thm]{Corollary}

\newtheorem{lem}[thm]{Lemma}
\newtheorem{prop}[thm]{Proposition}

\newtheorem{ex}[thm]{Example}

\newtheorem{rem}[thm]{Remark}

\newcommand{\ord}{v}
\newcommand{\p}{\mathbf{P}}

\newcommand{\Q}{\mathbf{Q}}

\newcommand{\Z}{\mathbf{Z}}

\newcommand{\OO}{\mathcal{O}}
\newcommand{\DD}{\mathcal{D}}

\newcommand{\F}{\mathbf{F}}

\newcommand{\E}{\rene{fix me}}
\newcommand{\EE}{\mathcal{E}}
\newcommand{\hatEE}{\widehat{\mathcal{E}}}
\newcommand{\Hom}{\mathrm{Hom}}

\newcommand{\mm}{\mathfrak{m}}

\newcommand{\Spec}{\operatorname{Spec}}

\newcommand{\wht}{\operatorname{wt}}

\newcommand{\Ext}{\operatorname{Ext}}

\newcommand{\sm}{\operatorname{sm}}

\newcommand{\cp}{\F_p}

\newcommand{\PGL}{\textnormal{PGL}}

\begin{document}

\title{On $p$-torsion of $p$-adic elliptic curves with additive reduction}
\author{Ren\'e Pannekoek}
\date{\today}

\maketitle

\section{Introduction}

\noindent
In this article, we fix a prime $p$. If $E/\Q_p$ is an elliptic curve with additive reduction, and one chooses for it a minimal Weierstrass equation over $\Z_p$:
$$
y^2 + a_1 xy + a_3 y = x^3 + a_2 x^2 + a_4 x + a_6,~~~~\textnormal{$a_i \in \Z_p$ for each $i$},
$$
then we denote by $E_0(\Q_p) \subset E(\Q_p)$ the subgroup of points that reduce to a non-singular point of the reduced curve. As is well-known, this subgroup does not depend on the choice of minimal Weierstrass equation.

The purpose of this note is to investigate the structure of $E_0(\Q_p)$ as a topological group. 

\begin{thm}
\label{main}
Let $E/\Q_p$ be an elliptic curve with additive reduction, such that it can be given by a minimal Weierstrass equation over $\Z_p$:
$$
y^2 + a_1 xy + a_3 y = x^3 + a_2 x^2 + a_4 x + a_6,
$$
where the $a_i$ are contained in $p\Z_p$ for each $i$. Then the group $E_0(\Q_p)$ is topologically isomorphic to $\Z_p$, except in the following four cases:
\begin{itemize}\itemsep=0pt
\item[(i)] $p=2$ and $a_1+a_3 \equiv 2 \pmod{4}$;
\item[(ii)] $p=3$ and $a_2 \equiv 6 \pmod{9}$;
\item[(iii)] $p=5$ and $a_4 \equiv 10 \pmod{25}$;
\item[(iv)] $p=7$ and $a_6 \equiv 14 \pmod{49}$.
\end{itemize}
In each of the cases (i)-(iv), $E_0(\Q_p)$ is topologically isomorphic to $p\Z_p \times \cp$, where $\cp$ has the discrete topology.
\end{thm} 

The proof of Theorem \ref{main} will be given in Section \ref{proofofmain}. The case $p>7$ of Theorem \ref{main} was also mentioned in \cite{SwD}.

We will say a few words about the idea of the proof. It is a standard fact from the theory of elliptic curves over local fields \cite[VII.6.3]{silverman} that $E_0(\Q_p)$ admits a canonical filtration
$$
E_0(\Q_p) \supset E_1(\Q_p) \supset E_2(\Q_p) \supset E_3(\Q_p) \supset \ldots,
$$
where for each $i \geq 1$ the quotient $E_i(\Q_p)/E_{i+1}(\Q_p)$ is isomorphic to $\cp$. The quotient $E_0(\Q_p)/E_{1}(\Q_p)$ is also isomorphic to $\cp$ by the fact that $E$ has additive reduction. One has a natural isomorphism of topological groups $j : E_2(\Q_p) \stackrel{\sim}{\rightarrow} p^2\Z_p$ given by the theory of formal groups. If $p>2$, the same theory even gives a natural isomorphism $j' : E_1(\Q_p) \stackrel{\sim}{\rightarrow} p\Z_p$. These isomorphisms identify $E_n(\Q_p)$ with $p^n\Z_p$ for all $n \geq 2$. The idea of the proof of theorem \ref{main} is to start from $j$ or $j'$ and, by extending its domain, to build up an isomorphism between $E_0(\Q_p)$ and either $\Z_p$ or $p\Z_p \times \cp$.

Rather than elliptic curves over $\Q_p$ with additive reduction, we consider the more general case of Weierstrass curves over $\Z_p$ whose generic fiber is smooth and whose special fiber is a cuspidal cubic curve. This allows more general results. Theorem \ref{main} is derived as a special case.

At the end of the note, we give examples for each prime $2 \leq p \leq 7$ of an elliptic curve $E/\Q$ with additive reduction at $p$ such that $E_0(\Q_p)$ contains a $p$-torsion point defined over $\Q$.

\section{Preliminaries}
\label{preliminaries}

\subsection{Preliminaries on Weierstrass curves}
\label{weier}

All proofs of facts recalled in this section can be found in \cite[Ch. IV, VII]{silverman}.

Let $K$ be a finite field extension of $\Q_p$ for some prime $p$, and let $v_K : K \rightarrow \Z \cup \{ \infty \}$ be its normalized valuation. Let $\OO_K$ be the ring of integers, $\mathfrak{m}_K$ its maximal ideal and $k$ its residue field. By a \defi{Weierstrass curve} over $\OO_K$ we mean a projective curve $\EE \subset \p^2_{\OO_K}$ defined by a Weierstrass equation
\begin{equation}
\label{Weiereq}
y^2 + a_1xy + a_3y = x^3+ a_2x^2 + a_4x + a_6,
\end{equation}
such that the generic fiber $\EE_K$ is an elliptic curve with $(0:1:0)$ as the origin. The coefficients $a_i$ are uniquely determined by $\EE$. The discriminant of $\EE$, denoted $\Delta_{\EE}$, is defined as in \cite[III.1]{silverman}. The curve $\EE$ is said to be minimal if $v_K(\Delta_{\EE})$ is minimal among $v_K(\Delta_{\EE'})$, where $\EE'$ ranges over the Weierstrass curves such that $\EE'_K \cong \EE_K$.

We will say that a Weierstrass curve $\EE/\OO_K$ has \defi{good reduction} when the special fiber $\EE_{k}$ is smooth, \defi{multiplicative reduction} when $\EE_{k}$ is nodal (i.e. there are two distinct tangent directions to the singular point), and \defi{additive reduction} when $\EE_{k}$ is cuspidal (i.e. one tangent direction to the singular point). A non-minimal Weierstrass curve has additive reduction.  The reduction type of an elliptic curve $E$ is defined to be the reduction type of a \defi{minimal Weierstrass model} of $E$ over $\OO_K$, which is a minimal Weierstrass curve $\EE/\OO_K$ such that $\EE_K \cong E$. By the fact that the minimal Weierstrass model of $E$ is unique up to $\OO_K$-isomorphism, this is well-defined.

We have $E(K) = \EE(K) = \EE(\OO_K)$ since $\EE$ is projective. Therefore, we have a reduction map $E(K) \rightarrow \EE(k)$ given by restricting an element of $\EE(\OO_K)$ to the special fiber. By $\EE_0(K)$ we denote the subgroup $\EE_0(K) \subset \EE(K)$ of points reducing to a non-singular point of the special fiber $\EE_k$. By $\EE_1(K) \subset \EE_0(K)$ we denote the \defi{kernel of reduction}, i.e. the points that map to the identity $0_{k}$ of $\EE(k)$. A more explicit definition of $\EE_1(K)$ is
\begin{equation}
\label{defE1}
\EE_1(K) = \left\{ (x,y) \in \EE(K) : v_K(x) \leq -2, v_K(y) \leq -3 \right\} \cup \{ 0_E \}.
\end{equation}
More generally, one defines subgroups $\EE_n(K) \subset \EE_0(K)$ as follows:
$$
\EE_n(K) = \left\{ (x,y) \in \EE(K) : v_K(x) \leq -2n, v_K(y) \leq -3n \right\} \cup \{ 0_E \}.
$$
We thus have an infinite filtration on the subgroup $\EE_1(K)$:
\begin{equation}
\label{filtE}
\EE_1(K) \supset \EE_2(K) \supset \EE_3(K) \supset \cdots
\end{equation}
For an elliptic curve $E/K$ and an integer $n \geq 0$, we define $E_n(K)$ to be $\EE_n(K)$, where $\EE$ is a minimal Weierstrass model of $E$ over $\OO_K$. The $E_n(K)$ are well-defined, again by the fact that the minimal Weierstrass model of $E$ is unique up to $\OO_K$-isomorphism.
\begin{prop}
\label{hensel}
For $\EE$ a Weierstrass curve over $\Z_p$, there is an exact sequence
\begin{equation*}
0 \rightarrow \EE_1(K) \rightarrow \EE_0(K) \rightarrow \wt{\EE}_{\textnormal{sm}}(k) \rightarrow 0,
\end{equation*}
where $\wt{\EE}_{\textnormal{sm}}$ is the complement of the singular points in the special fiber $\wt{\EE}$.
\end{prop}
\begin{proof}
This comes down to Hensel's lemma. See \cite[VII.2.1]{silverman}.
\end{proof}

For any Weierstrass curve $\EE$, we can consider its \defi{formal group} $\widehat{\EE}$ \cite[IV.1--2]{silverman}. This is a one-dimensional formal group over $\OO_K$. Giving the data of this formal group is the same as giving a power series $F = F_{\widehat{\EE}}$ in $\OO_K[[X,Y]]$, called the \defi{formal group law}. It satisfies
$$
F(X,Y) = X + Y + (\textnormal{terms of degree $\geq 2$})
$$
and
$$F(F(X,Y),Z)) = F(X,F(Y,Z)).
$$
For $\EE$ as in (\ref{Weiereq}), the first few terms of $F$ are given by:
\begin{align*}
& F(X,Y) \,= \\
& X+Y \, - \, a_1XY \, - \, a_2(X^2Y+XY^2)  \,-\,2a_3(X^3Y + XY^3) + (a_1a_2 - 3a_3)X^2Y^2 \,-\\
& (2a_1a_3+2a_4)(X^4Y+XY^4)-(a_1a_3-a_2^2+4a_4)(X^3Y^2+X^2Y^3) \,+ \ldots
\end{align*}
Treating the Weierstrass coefficients $a_i$ as unknowns, we may consider $F$ as an element of $\Z[a_1,a_2,a_3,a_4,a_6][[X,Y]]$ called the \defi{generic formal group law}. If we make $\Z[a_1,a_2,a_3,a_4,a_6]$ into a weighted ring with weight function wt, such that $\wht(a_i)=i$ for each $i$, then the coefficients of $F$ in degree $n$ are homogeneous of weight $n-1$ \cite[IV.1.1]{silverman}. For each $n \in \Z_{\geq 2}$, we define power series $[n]$ in $\OO_K[[T]]$ by $[2](T) = F(T,T)$ and $[n](T) = F([n-1](T),T)$ for $n\geq 3$. Here also, we may consider each $[n]$ either as a power series in $\OO_K[[T]]$ or as a power series in $\Z[a_1,a_2,a_3,a_4,a_6][[T]]$ called the \defi{generic multiplication by $n$ law}. We have:
\begin{lem}
\label{mult_by_p}
Let $[p] = \sum_n b_n T^n \in \Z[a_1,a_2,a_3,a_4,a_6][[T]]$ be the generic formal multiplication by $p$ law. Then:
\begin{enumerate}\itemsep=0pt
\item \label{p_deelt_bn} $p \mid b_n$ for all $n$ not divisible by $p$;
\item \label{gewicht_bn} $\textnormal{wt}(b_n) = n-1$, considering $\Z[a_1,a_2,a_3,a_4,a_6]$ as a weighted ring as above.\end{enumerate}
\end{lem}
\begin{proof}
(\ref{p_deelt_bn}) is proved in \cite[IV.4.4]{silverman}. (\ref{gewicht_bn}) follows from \cite[IV.1.1]{silverman} or what was said above.
\end{proof}
The series $F(u,v)$ converges to an element of $\mm_K$ for all $u,v \in \mm_K$. To $\EE$ one associates the group $\hatEE(\mathfrak{m}_K)$, the $\mathfrak{m}_K$-valued points of $\widehat{\EE}$, which as a set is just $\mathfrak{m}_K$, and whose group operation $+$ is given by $u+v = F(u,v)$ for all $u,v \in \hatEE(\mathfrak{m}_K)$. The identity element of $\hatEE(\mathfrak{m}_K)$  is $0 \in \mm_K$. If $n \geq 1$ is an integer, then by $\hatEE(\mathfrak{m}^n_K)$ we denote the subset of $\hatEE(\mathfrak{m}_K)$ corresponding to the subset $\mm_K^n \subset \mm_K$, where $\mm_K^n$ is the $n$th power of the ideal $\mm_K$ of $\OO_K$. The groups $\hatEE(\mathfrak{m}^n_K)$ are subgroups of $\hatEE(\mathfrak{m}_K)$, and we have an infinite filtration of $\widehat{\EE}(\mm_K)$:
\begin{equation}
\label{filtEhat}
\widehat{\EE}(\mm_K) \supset \widehat{\EE}(\mm_K^2) \supset \widehat{\EE}(\mm_K^3) \supset \cdots
\end{equation}
\begin{prop}\label{fund_isom}
The map
\begin{align*}
\psi_K: \EE_1(K) & \stackrel{\sim}{\rightarrow} \widehat{\EE}(\mathfrak{m}_K) \\
(x,y) & \mapsto -x/y \\
0 & \mapsto 0
\end{align*}
is a isomorphism of topological groups. Moreover, $\psi_K$ respects the filtrations (\ref{filtE}) and (\ref{filtEhat}), i.e. it identifies the subgroups $\EE_n(K)$ defined above with $\hatEE(\mathfrak{m}^n_K)$.
\end{prop}
\begin{proof}
See \cite[VII.2.2]{silverman}. 
\end{proof}
It follows from the proof given in \cite[VII.2.2]{silverman} that there exists a power series $w\in\OO_K[[T]]$, with the first few terms given by
$$
w(T) = T^3 + a_1 T^4 + (a_1^2+a_2) T^5 + (a_1^3+2a_1a_2+a_3)T^6 + \ldots,
$$
such that the inverse to $\psi_K$ is given by $z \mapsto (z/w(z),-1/w(z))$. Given a finite field extension $K \subset L$, we have an obvious commutative diagram
\[
\xymatrix{
\EE_1(K) \ar[r]^{\psi_K}   \ar[d]^{\textnormal{incl}} & \widehat{\EE}(\mathfrak{m}_K) \ar[d]^{\textnormal{incl}}   \\\
\EE_1(L) \ar[r]^{\psi_L}                            & \widehat{\EE_{\OO_L}}(\mathfrak{m}_L)       
}
\]
Here $\widehat{\EE_{\OO_L}}(\mathfrak{m}_L)$ is the set of $\mathfrak{m}_L$-valued points of the formal group of $\EE_{\OO_L}$, the base-change of $\EE$ to $\Spec(\OO_L)$.


\subsection{Extensions of topological abelian groups}
\begin{prop}
\label{extensions}
Suppose $X$ is a topological abelian group and we have a short exact sequence
$$
0 \rightarrow \Z_p^d \rightarrow X \rightarrow \cp \rightarrow 0.
$$
of topological groups where the second arrow is a topological embedding. Then $X$ is isomorphic as a topological group to either $\Z_p^d$ or $\Z_p^d \times \cp$.
\end{prop}
It is indeed necessary to require $\Z_p^d \rightarrow X$ to be a topological embedding, i.e. a homeomorphism onto its image, since otherwise we could take $X$ to be the product $(\Z_p^d)^{\textnormal{ind}} \times \cp$, where the first factor is the abelian group $\Z_p^d$ endowed with the indiscrete topology.
\begin{proof}
First, we claim that $\Ext^1_{\Z}(\cp,A) = A/pA$ for any abelian group $A$. Taking the long exact sequence associated to $\Hom_{\Z}(-,A)$ for the exact sequence $0 \rightarrow \Z \stackrel{p}{\rightarrow} \Z \rightarrow \cp \rightarrow 0$ results in the exact sequence
$$
\Hom(\Z,A) \stackrel{p}{\rightarrow} \Hom(\Z,A) \rightarrow \Ext^1_{\Z}(\cp,A) \rightarrow \Ext^1_{\Z}(\Z,A) = 0
$$
where the last equality follows from the fact that $\Hom(\Z,-)$ is an exact functor. Using that $\Hom(\Z,A) = A$, we get
$$
\Ext^1_{\Z}(\cp,A) = A/pA,
$$
which proves the claim. Putting $A = \Z_p^d$, we find $\Ext^1_{\Z}(\cp,\Z_p^d)=\cp^d$. We conclude that the equivalence classes of extensions of $\Z$-modules $0 \rightarrow \Z_p^d \rightarrow X \rightarrow \cp \rightarrow 0$ are in bijective correspondence with the elements of $\cp^d$. The element $0 \in \cp^d$ corresponds to the split extension. The non-split ones are obtained as follows. For $v \in \Z_p^d - p\Z_p^d$, we construct an extension 
$$
0 \rightarrow \Z_p^d \rightarrow X_v \stackrel{f_v}{\rightarrow} \cp \rightarrow 0.
$$
by defining the subgroup $X_v \subset \Q_p^d$ as $X_v = \Z_p^d + \left\langle v/p \right\rangle$ and letting $f_v : X_v \rightarrow \cp$ be the unique group homomorphism that is trivial on $\Z_p^d \subset X_v$ and that sends $v/p$ to $1$. The equivalence class of the above extension only depends on the class of $v$ modulo $p \Z_p^d$. Note that if we take any element $x \in X_v$ mapping to $1\in \cp$, we have $px = v+ pv_1 \in \Z_p^d$ for some $v_1 \in \Z_p^d$. Note further that $X_v$ is topologically isomorphic to $\Z_p^d$, if we give it the subspace topology.

A diagram chase shows that this construction gives us $p^d-1$ different equivalence classes of extensions. Suppose that $v,w \in \Z_p^d - p\Z_p^d$ and $\phi : X_v \stackrel{\sim}{\rightarrow} X_w$ are such that
\[
\xymatrix{
0 \ar[r] & \Z_p^d \ar[r] \ar[d]^{\textnormal{id}} & X_v \ar[r]^{f_v} \ar[d]^{\phi} & \cp \ar[r] \ar[d]^{\textnormal{id}} & 0  \\\
0 \ar[r] & \Z_p^d \ar[r] 				  & X_w \ar[r]^{f_w} & \cp \ar[r]                                          & 0 
}
\]
is a commutative diagram. Consider an element $x \in X_v$ such that $f_v(x) = 1$. Then $f_w(\phi(x))= 1$. Furthermore, $px = v + pv_1$ for some $v_1 \in p\Z_p^d$, and $\phi(px) = p \phi(x) = w + pw_1$ for some $w_1 \in p\Z_p^d$. Hence $v + pv_1 = \phi(v+pv_1) = w + pw_1$, so $v \equiv  w \pmod{p\Z_p^d}$.


Let $X$ be a topological group sitting inside an extension of topological groups $0 \rightarrow \Z_p^d \stackrel{i}{\rightarrow} X \stackrel{f}{\rightarrow} \cp \rightarrow 0$, with $i$ a topological embedding and $f$ continuous. This means that there exists an extension of topological groups $0 \rightarrow \Z_p^d \rightarrow Y \rightarrow \cp \rightarrow 0$ that is either split or equal to one of the form $0 \rightarrow \Z_p^d \rightarrow X_v \stackrel{f_v}{\rightarrow} \cp \rightarrow 0$, an isomorphism of groups $\phi : X \stackrel{\sim}{\rightarrow} Y$, and a commutative diagram:
\[
\xymatrix{
0 \ar[r] & \Z_p^d \ar[r] \ar[d]^{\textnormal{id}} & X \ar[r]^{f} \ar[d]^{\phi} & \cp \ar[r] \ar[d]^{\textnormal{id}} & 0  \\\
0 \ar[r] & \Z_p^d \ar[r] 				  & Y \ar[r] & \cp \ar[r]                                          & 0.
}
\]
We claim that $\phi$ must also be a homeomorphism. Since both $X$ and $Y$ are topological disjoint unions of the translates of their subgroups $\Z_p^d$, and $\phi$ respects the disjoint union decomposition, this is clear. So $X$ is topologically isomorphic to $Y$, and hence to either $\Z_p^d$ or $\Z_p^d \times \cp$.
\end{proof}

\begin{rem}\label{induction}
\upshape
By repeatedly applying Proposition \ref{extensions}, we see that if we have a finite filtration
$$
\Z_p^d = M_n \subset M_{n-1} \subset  \ldots \subset M_1
$$
of topological groups, in which all quotients are isomorphic to $\cp$, then $M_1$ is torsion-free if and only if it is topologically isomorphic to $\Z_p^d$.
\end{rem}

The following is a strengthening of Proposition \ref{extensions} in the case $d=1$, which will be important for us.

\begin{cor}
\label{direct}
Let $p$ be a prime and suppose we have a short exact sequence
$$
0 \rightarrow p \Z_p \stackrel{i}{\rightarrow} X \rightarrow \cp \rightarrow 0
$$
of topological abelian groups where the second arrow is a topological embedding. If $X$ is topologically isomorphic to $\Z_p$, then $v_p(i^{-1}(px)) = 1$ for all $x \in X - i(p\Z_p)$, where $v_p$ is the $p$-adic valuation. If $X$ is not topologically isomorphic to $\Z_p$, it is topologically isomorphic to $p\Z_p \times \F_p$, and we have $v_p(i^{-1}(px)) > 1$ for all $x \in X - i(p\Z_p)$.
\end{cor}
\begin{proof}
If $X$ is topologically isomorphic to $\Z_p$, the map $i$ is given by multiplication by some unit $\alpha \in \Z_p^{\ast}$ followed by the inclusion $p \Z_p \subset \Z_p$. The conclusion follows.

If $X$ is not topologically isomorphic to $\Z_p$, then by Proposition \ref{extensions} we must have $X \cong p\Z_p \times \F_p$. But then if $x = (y,c)$, we have $v_p(i^{-1}(px)) = v_p(py) > 1$.
\end{proof}

\begin{lem}
Let $K$ be a finite extension of $\Q_p$ with ring of integers $\OO_K$. Then $\OO_K$ is topologically isomorphic to $\Z_p^d$, where $d = [K:\Q_p]$.
\end{lem}
\begin{proof}
$\OO_K$ is a free $\Z_p$-module of rank $d$, so there is a group isomorphism $\Z_p^d \stackrel{\sim}{\rightarrow} \OO_K$. Since both groups are topologically finitely generated, any isomorphism between them is bicontinuous \cite[1.1]{progroups}.
\end{proof}

\section{Weierstrass curves with additive reduction over $\OO_K$}


As in section \ref{preliminaries}, let $K$ be a finite extension of $\Q_p$. Let $\OO_K$ again be the ring of integers of $K$, with maximal ideal $\mathfrak{m}_K$ and residue field $k$.

In this section, we gather some general properties of Weierstrass curves over $\OO_K$ with additive reduction.
\begin{lem}
\label{ai_in_maxideaal}
Let $\EE/\OO_K$ be a Weierstrass curve with additive reduction. Then $\EE$ is $\OO_K$-isomorphic to a Weierstrass curve of the form
$$
y^2 + a_1 xy + a_3 y = x^3 + a_2 x^2 + a_4 x + a_6,
$$
where all $a_i$ lie in $\mathfrak{m}_K$.
\end{lem}
\begin{proof}
We construct an automorphism $\alpha \in \PGL_3(\OO_K)$ that maps $\EE$ to a Weierstrass curve of the desired form. Consider a translation $\alpha_1 \in \PGL_3(\OO_K)$ moving the singular point of the special fiber $\EE_k$ to $(0:0:1)$. The image $\EE_1 = \alpha_1(\EE)$ is a Weierstrass curve with coefficients satisfying $a_3,a_4,a_6$ in $\mm_K$. There exists a second automorphism $\alpha_2 \in \PGL_3(\OO_K)$, of the form $x' = x, y' = y + cx$, such that in the special fiber of $\alpha_2(\EE_1)$ the unique tangent at $(0:0:1)$ is given by $y'=0$. The Weierstrass curve $\EE_2=\alpha_2(\EE_1)$ now has all its coefficients $a_1,a_2,a_3,a_4,a_6$ in $\mm_K$. One may thus take $\alpha = \alpha_2 \circ \alpha_1$.
\end{proof}

Suppose that $\EE/\OO_K$ is a Weierstrass curve given by (\ref{Weiereq}), and suppose that the $a_i$ are contained in $\mm_K$. In particular, $\EE$ has additive reduction. If we let $F$ denote the formal group law of $\EE$, then the assumption on the $a_i$ implies that $F(u,v)$ converges to an element of $\OO_K$ for all $u,v \in \OO_K$. Hence $F$ can be seen to induce a group structure on $\OO_K$, extending the group structure on $\widehat{\EE}(\mm_K)$. The same statement holds true when we replace $K$ by a finite field extension $L$.
\begin{defn}
\upshape
Let $\EE/\OO_K$ be a Weierstrass curve given by (\ref{Weiereq}), and assume that the $a_i$ are contained in $\mm_K$. For any finite field extension $K \subset L$, we denote by $\widehat{\EE}(\OO_L)$ the topological group obtained by endowing the space $\OO_L$ with the group structure induced by $F$.
\end{defn}

The following proposition will be fundamental in determining of the structure of $\EE_0(\Q_p)$ as a topological group for Weierstrass curves with additive reduction.

\begin{prop} 
\label{tocomputewith}
Let $\EE/\OO_K$ be a Weierstrass curve given by (\ref{Weiereq}), and assume that the $a_i$ are contained in $\mm_K$. 
\begin{enumerate}\itemsep=0pt
\item \label{Fhomeo} The map $\Psi : \EE_0(K) \rightarrow \widehat{\EE}(\OO_K)$ that sends $(x,y)$ to $-x/y$ is an isomorphism of topological groups.
\item \label{pgt7} If $6e(K/\Q_p) < p-1$, where $e$ denotes the ramification degree, then $\EE_0(K)$ is also topologically isomorphic to $\OO_K$ equipped with the usual group structure.
\end{enumerate}
\end{prop}
\begin{proof}
Let $\pi$ be a uniformizer for $\OO_K$. Consider the field extension $L = K(\rho)$ with $\rho^6 = \pi$. Then define the Weierstrass curve $\DD$ over $\OO_L$ by
$$
y^2 + \alpha_1 xy + \alpha_3 y = x^3 + \alpha_2 x^2 + \alpha_4 x^4 x + \alpha_6,
$$
where $\alpha_i = a_i / \rho^i$. There is a birational map $ \phi : \EE \times_{\OO_K} \OO_L \dashrightarrow \DD$, given by $\phi(x,y)= (x/\rho^2,y/\rho^3)$. The birational map $\phi$ induces an isomorphism on generic fibers, and hence a homeomorphism between $\EE(L)$ and $\DD(L)$. Using (\ref{defE1}) and the fact that we have $(x,y)\in \EE_0(L)$ if and only if $v_L(x),v_L(y)$ are both not greater than zero, one sees that $\phi$ induces a bijection $\EE_0(L) \stackrel{\sim}{\rightarrow} \DD_1(L)$, that all maps (\textit{a priori} just of sets) in the following diagram are well-defined, and that the diagram commutes:
\[
\xymatrix{
\EE_1(K) \ar[d]^{\psi_K}  \ar[r]^{\textnormal{incl}} & 
\EE_0(K) \ar[d]^{\Psi} \ar[r]^{\textnormal{incl}} &                  
\EE_0(L) \ar[d]^{\Psi_L}  \ar[r]^{\phi}       &            
\DD_1(L) \ar[d]^{\psi_L}                            \\\
\widehat{\EE}(\mm_K)  \ar[r]^{\textnormal{incl}} & \widehat{\EE}(\OO_K)  \ar[r]^{\textnormal{incl}} & \widehat{\EE}(\OO_L)  \ar[r]^{\cdot \rho} & \widehat{\DD}(\mm_L) 
}
\]
Here the map $\Psi_{L} : \EE_0(L) \rightarrow \OO_{L}$ is defined by $(x,y) \mapsto -x/y$, the rightmost lower horizontal arrow is multiplication by $\rho$, and the maps labeled $\,\textnormal{incl}\,$ are the obvious inclusions. Note that the horizontal and vertical outer maps are all continuous. Since $\psi_L$, $\phi$ and multiplication by $\rho$ are homeomorphisms (for $\psi_L$ one uses Proposition \ref{fund_isom}), so is $\Psi_L$. Hence $\Psi$ must be a homeomorphism onto its image. By Galois theory, $\Psi$ is surjective, so it is itself a homeomorphism.

Let $F_{\widehat{\DD}}$ be the formal group law of $\DD$. One calculates that
$$
\rho F(X,Y) = F_{\widehat{\DD}}(\rho X, \rho Y).
$$
Hence all maps in the diagram are group homomorphisms. This proves the first part of the proposition.

Now assume $6e(K/\Q_p) < p-1$, so that $v_L(p) = 6 v_K(p) = 6e(K/\Q_p)< p-1$. Now \cite[IV.6.4(b)]{silverman} implies that $\EE_1(K)$ is topologically isomorphic to $\mm_K$, and $\DD_1(L)$ to $\mm_L$. Since $\EE$ has additive reduction, we have $\wt{\EE}_{\sm}(k) \cong k^+ \cong \F_p^f$, where $f = f(K/\Q_p)$ is the inertia degree of $K/\Q_p$ and $\wt{\EE}_{\sm}$ is the smooth locus of the special fiber of $\EE$. Proposition \ref{hensel} shows we have a short exact sequence
$$
0 \rightarrow \mm_K \rightarrow \EE_0(K) \rightarrow \F_p^f \rightarrow 0.
$$
In the diagram above, the topological group $\EE_0(K)$ is mapped homomorphically into the torsion-free group $\DD_1(L)$, hence it is itself torsion-free. It follows from Remark \ref{induction} that $\EE_0(K)$ is topologically isomorphic to $\OO_K$. This proves the second part.
\end{proof}
The following corollary is worth noting, but will not be used in what follows.
\begin{cor}
Let $\EE/\OO_K$ be a Weierstrass curve with additive reduction. If $6e(K/\Q_p)<p-1$, then $\EE_0(K)$ is topologically isomorphic to $\OO_K$.
\end{cor}
\begin{proof}
The statement that $\EE_0(K)$ is topologically isomorphic to $\OO_K$ only depends on the $\OO_K$-isomorphism class of $\EE$. By Lemma \ref{ai_in_maxideaal}, there exists a Weierstrass curve $\EE'$ with $a_i \in \mm_K$ that is $\OO_K$-isomorphic to $\EE$. Now apply Proposition \ref{tocomputewith} to $\EE'$. 
\end{proof}

\section{Weierstrass curves with additive reduction over $\Z_p$}

In this section, we gather some general properties of Weierstrass curves over $\Z_p$ with additive reduction and finish the proof of theorem \ref{main}.
\begin{lem}
\label{E1}
Let $\EE/\Z_p$ be a Weierstrass curve with additive reduction. Then there exists a topological isomorphism $\chi : \widehat{\EE}(p \Z_p) \stackrel{\sim}{\rightarrow} p\Z_p$ such that for $n \in \Z_{\geq 1}$, $\chi$ identifies $\widehat{\EE}(p^n \Z_p)$ with $p^n \Z_p$.
\end{lem}
\begin{proof}
For $p>2$, this is standard; the proof may be found in \cite[IV.6.4(b)]{silverman}. We now treat the case $p=2$. By Lemma \ref{ai_in_maxideaal}, we may assume that the Weierstrass coefficients $a_i$ of $\EE$ all lie in $2\Z_2$. The multiplication by 2 on $\widehat{\EE}(2\Z_2)$ is given by the power series
\begin{equation}
\label{formalduplication}
[2](T) = F_{\widehat{\EE}}(T,T) = 2T - a_1T^2 -2a_2T^3 + (a_1a_2-7a_3)T^4 - \ldots,
\end{equation}
where $F_{\widehat{\EE}}$ is the formal group law of $\EE$. By \cite[IV.3.2(a)]{silverman}, $\widehat{\EE}(2\Z_2)/\widehat{\EE}(4\Z_2)$ is cyclic of order 2. By \cite[IV.6.4(b)]{silverman}, there exists a topological isomorphism $\widehat{\EE}(4\Z_2) \stackrel{\sim}{\rightarrow} 4\Z_2$. Hence there exists an extension
$$
0 \rightarrow 4 \Z_2 \stackrel{i}{\rightarrow} \widehat{\EE}(2\Z_2) \rightarrow \F_2 \rightarrow 0.
$$
From Theorem \ref{extensions} we see that $\widehat{\EE}(2\Z_2)$ is topologically isomorphic either to $2\Z_2$ or to $4\Z_2 \times \F_2$. Assume that the latter is the case, then there is an element $z$ of order 2 in $\widehat{\EE}(2\Z_2)$ that is not contained in $\widehat{\EE}(4\Z_2)$. For such a $z$ we have $v_2(z) = 1$, where $v_2 : \widehat{\EE}(2\Z_2) \rightarrow \Z_{\geq 1} \cup \{ \infty \}$ is the 2-adic valuation on the underlying set $2\Z_2$ of $\widehat{\EE}(2\Z_2)$. Using that in the duplication power series (\ref{formalduplication}) we have $a_i \in 2\Z_2$ for each $i$, it follows that $v_2([2](z)) = 2$, so $[2](z) \neq 0$. This is a contradiction, so there exists an isomorphism $\chi : \widehat{\EE}(2\Z_2) \stackrel{\sim}{\rightarrow} 2\Z_2$ as topological groups. From this, and from the fact that $\widehat{\EE}(2^n\Z_2)/\widehat{\EE}(2^{n+1}\Z_2) \cong \F_2$ for all $n \in \Z_{\geq 1}$ \cite[IV.3.2(a)]{silverman}, we see that $\chi$ necessarily respects the filtrations on either side.
\end{proof}

\begin{cor}
\label{E1cor}
Let $\EE/\Z_p$ be a Weierstrass curve with additive reduction. Then there exists an isomorphism $\EE_1(\Q_p) \stackrel{\sim}{\rightarrow} p\Z_p$ which for $n \in \Z_{\geq 1}$ identifies $\EE_n(\Q_p)$ with $p^n \Z_p$.
\end{cor}
\begin{proof}
Such an isomorphism can be obtained by composing the isomorphism $\chi$ from Lemma \ref{E1} with the isomorphism $\psi_{\Q_p}$ from Proposition \ref{fund_isom}.
\end{proof}

\subsection{$p=2$}
\begin{prop}
\label{case2}
Let $\EE/\Z_2$ be a Weierstrass curve with its coefficients $a_i$ in $2\Z_2$. Then $\EE_0(\Q_2)$ is topologically isomorphic to $\Z_2$ if $a_1 + a_3 \equiv 0 \pmod{4}$, and to $2\Z_2 \times \F_2$ otherwise.
\end{prop}
\begin{proof}
Proposition \ref{hensel} shows that there is a short exact sequence
$$
0 \rightarrow \EE_1(\Q_2) \rightarrow \EE_0(\Q_2) \rightarrow \F_2 \rightarrow 0.
$$
By Lemma \ref{E1}, we have $\EE_1(\Q_2) \cong 2\Z_2$, so Proposition \ref{extensions} implies that $\EE_0(\Q_2)$ is topologically isomorphic either to $\Z_2$ or to $2 \Z_2 \times \F_2$. 

Let $[2](T) \in \OO_K[[T]]$ be the formal duplication formula (\ref{formalduplication}) on $\EE$. Let $\Psi$ be the map from Proposition \ref{tocomputewith}. Since $\Psi$ is an isomorphism of topological groups, we have for all $P \in \EE_0(\Q_2)$:
\begin{equation}
\label{truth2}
\Psi(2P) = [2](\Psi(P)).
\end{equation}
By Corollary \ref{direct}, we have $\EE_0(\Q_2) \cong \Z_2$ if and only if for all $P \in \EE_0(\Q_2) - \EE_1(\Q_2)$ we have $2P = \EE_1(\Q_2) - \EE_2(\Q_2)$, which by (\ref{truth2}) is true if and only if for all $z \in \widehat{\EE}(\Z_2) - \widehat{\EE}(2\Z_2)$ we have $v_2([2](z)) = 1$, where $v_2 : \widehat{\EE}(\Z_2) \rightarrow \Z_{\geq 0} \cup \{ \infty \}$ is the 2-adic valuation on the underlying set $\Z_2$ of $\widehat{\EE}(\Z_2)$. This condition may be checked using the duplication power series
$$
[2](T) = 2T - a_1T^2 -2a_2T^3 + (a_1a_2-7a_3)T^4 - \ldots = \sum_{i=1}^{\infty} b_i T^i.
$$
In deciding whether $v_2([2](z)) = 1$ for $z \in \widehat{\EE}(\Z_2) - \widehat{\EE}(2\Z_2)$, we do not need to consider those parts of terms whose coefficients have valuation $\geq 2$. The non-linear parts of each coefficient $b_i$ will contribute only terms with valuation $\geq 2$, so may ignore these and keep only the linear parts. The terms $b_i z^i$ with $i$ odd we may discard altogether; by Lemma \ref{mult_by_p}, all their coefficients have valuation $\geq 2$. Finally, we may discard all terms $b_i z^i$ with $i$ even and $\geq 6$: a polynomial in $\Z[a_1,\ldots,a_6]$ whose weight is odd and at least 5 does not contain a linear term (there being no $a_5$), so the terms involving $z^6, z^8, z^{10}, \ldots$ will have valuation $\geq 2$. 

We thus get that, if $z \in \widehat{\EE}(\Z_2) - \widehat{\EE}(2\Z_2)$,
$$
v_2([2](z)) = 1 ~~~ \Leftrightarrow ~~~ v_2(2z - a_1z^2 - 7a_3z^4) = 1.
$$
This is true for all $z \in \widehat{\EE}(\Z_2) - \widehat{\EE}(2\Z_2)$ if and only if:
$$
v_2(z - \frac{a_1}{2}z^2 - \frac{7a_3}{2}z^4) = 0  ~~~\Leftrightarrow ~~~ a_1 + 7a_3 \equiv 0 \pmod{4}   ~~~\Leftrightarrow ~~~ a_1 + a_3 \equiv 0 \pmod{4}
$$
since $z \equiv z^2 \equiv z^4 \pmod{2}$. This proves the proposition.
\end{proof}

\subsection{$p=3$}
\begin{prop}
\label{case3}
Let $\EE/\Z_3$ be a Weierstrass curve with its coefficients $a_i$ in $3\Z_3$. Then $\EE_0(\Q_3)$ is topologically isomorphic to $\Z_3$ if $a_2 \not \equiv 6 \pmod{9}$, and to $3\Z_3 \times \F_3$ otherwise.
\end{prop}
\begin{proof}
We proceed as in the proof of Proposition \ref{case2}, using the formal triplication formula:
\begin{equation}
\label{formaltriplication}
[3](T) = 3T - 3a_1T^2 +(a_1^2-8a_2)T^3 + (12a_1a_2 - 39a_3)T^4 + \ldots = \sum_{i=1}^{\infty} b_i T^i.
\end{equation}
We consider the usual exact sequence for $\EE_0(\Q_3)$:
$$
0 \rightarrow \EE_1(\Q_3) \rightarrow \EE_0(\Q_3) \rightarrow \F_3 \rightarrow 0.
$$
We see from $\EE_1(\Q_3) \cong 3\Z_3$ and Corollary \ref{direct} that $\EE_0(\Q_3)$ is topologically isomorphic to $3\Z_3 \times \F_3$ if and only if for all elements $z \in \widehat{\EE}(\Z_3) - \widehat{\EE}(3\Z_3)$, $[3](z)$ has valuation greater than 1. On the other hand, $\EE_0(\Q_3)$ is topologically isomorphic to $\Z_3$ if for all such $z$, the valuation of $[3](z)$ is 1. Reasoning as in the proof of Proposition \ref{case2}, we see that we may ignore all terms of degree not equal to 1 or a multiple of 3 since their coefficients are divisible by 3 and have positive weight. Also we may ignore the terms of degree both equal to a multiple of 3 and greater than 3, since their coefficients do not contain parts that are linear in $a_1,\ldots,a_6$. Finally, we may ignore the non-linear part of the term of degree 3. We see that for $z \in \widehat{\EE}(\Z_3) - \widehat{\EE}(3\Z_3)$, we have:
$$
v_3([3](z)) = 1 ~~~\Leftrightarrow ~~~ v_3(3z - 8a_2z^3) = 1.
$$
This happens for all such $z$ if and only if:
$$
\ord_3(z - \frac{8a_2}{3}z^3 ) = 0  ~~~\Leftrightarrow ~~~ 1 - \frac{8a_2}{3} \not \equiv 0 \pmod{3} ~~~\Leftrightarrow ~~~ a_2 \not \equiv 6 \pmod{9}
$$
since $z \equiv z^3  \pmod{3}$. This proves the proposition.
\end{proof}

\subsection{$p=5$}
\begin{prop}
\label{case5}
Let $\EE/\Z_5$ be a Weierstrass curve with its coefficients $a_i$ in $5\Z_5$. Then $\EE_0(\Q_5)$ is topologically isomorphic to $\Z_5$ if $a_4 \not \equiv 10 \pmod{25}$, and to $5\Z_5 \times \F_5$ otherwise.
\end{prop}
\begin{proof}
For simplicity, we give the formal multiplication by 5 power series in the case where $a_1,a_2,a_3$ are zero:
\begin{equation}
\label{formmultby5}
[5](T) = 5T - 1248a_4 T^5 + \ldots = \sum_{i=1}^{\infty} b_i T^i
\end{equation}
This formula suffices for our purposes, since the same arguments as in the proofs of Propositions \ref{case2} and \ref{case3} show that the terms that are canceled by setting $a_1=a_2=a_3=0$ could have been ignored anyway.

We apply Corollary \ref{direct} to:
$$
0 \rightarrow 5\Z_5 \rightarrow \EE_0(\Q_5) \rightarrow \F_5 \rightarrow 0.
$$
In (\ref{formmultby5}) we may ignore terms of degree not equal to 1 or 5, by the same reasoning as in the proofs of Propositions \ref{case2} and \ref{case3}. We see that for $z \in \widehat{\EE}(\Z_5) - \widehat{\EE}(5\Z_5)$ we have:
$$
\ord_5([5](z)) = 1 ~~~\Leftrightarrow ~~~ \ord_5(5z - 1248a_4z^5) = 1.
$$
This happens for all such $z$ if and only if:
$$
\ord_5(z - \frac{1248a_4}{5}z^5 ) = 0  ~~~\Leftrightarrow ~~~ 1 - \frac{1248a_4}{5} \not \equiv 0 \pmod{5} ~~~\Leftrightarrow ~~~ a_4 \not \equiv 10 \pmod{25}
$$
since $z \equiv z^5  \pmod{5}$. This proves the proposition.
\end{proof}

\subsection{$p=7$}
\begin{prop}
\label{case7}
Let $\EE/\Z_7$ be a Weierstrass curve with its coefficients $a_i$ in $7\Z_7$. Then $\EE_0(\Q_7)$ is topologically isomorphic to $\Z_7$ if $a_6 \not \equiv 14 \pmod{49}$, and to $7\Z_7 \times \F_7$ otherwise.
\end{prop}
\begin{proof}
For simplicity, we give the formal multiplication by 7 power series with $a_1,a_2,a_3$ set to zero:
\begin{equation}
\label{formmultby7}
[7](T) = 7T - 6720a_4T^5 - 352944a_6 T^7 + \ldots
\end{equation}
As before, the terms that have disappeared as a result could have been ignored anyway.

We apply Corollary \ref{direct} to:
$$
0 \rightarrow 7\Z_7 \rightarrow \EE_0(\Q_7) \rightarrow \F_7 \rightarrow 0,
$$
In (\ref{formmultby7}) we may ignore terms of degree not equal to 1 or 7, by the same reasoning as in the proofs of Propositions \ref{case2} and \ref{case3}. We see that for $z \in \widehat{\EE}(\Z_7) - \widehat{\EE}(7\Z_7)$ we have:
$$
\ord_7([7](z)) = 1 ~~~\Leftrightarrow ~~~ \ord_7(7z - 352944a_6z^7) = 1.
$$
This happens if and only if:
$$
\ord_7(z - \frac{352944a_6}{7}z^7 ) = 0  ~~~\Leftrightarrow ~~~ 1 - \frac{352944a_6}{7} \not \equiv 0 \pmod{7} ~~~\Leftrightarrow ~~~ a_6 \not \equiv 14 \pmod{49}
$$
since $z \equiv z^7  \pmod{7}$. This proves the proposition.
\end{proof}

\subsection{The proof of Theorem \ref{main}}
\label{proofofmain}
We are now ready to derive Theorem \ref{main} from our previous results.

Let $E/\Q_p$ and $a_1,\ldots,a_6 \in p \Z_p$ be as in the statement of the theorem. Then the Weierstrass curve 
$$
y^2 + a_1 xy + a_3 y = x^3 + a_2 x^2 + a_4 x + a_6,
$$
over $\Z_p$ defines a minimal Weierstrass model of $E$. The theorem follows by applying to $\EE$ part \ref{pgt7} of Proposition \ref{tocomputewith} if $p>7$, or one of Propositions \ref{case2}--\ref{case7} if $p \leq 7$.

\section{Examples}
\label{examples}

In this section, we have collected some examples of elliptic curves over $\Q_p$ with additive reduction, such that their points of good reduction contains a $p$-torsion point. In particular, all curves and torsion points are defined over $\Q$. The fact that they possess a $p$-torsion point of good reduction can be verified using the appropriate result from the previous section. (Note that these result do not say when the $p$-torsion points will be defined over $\Q$.)

\begin{ex}\upshape
The elliptic curve
$$
E_2 : y^2 - 2 y = x^3 - 2
$$
has additive reduction at 2, and its 2-torsion point $(1,1)$ is of good reduction. 
\end{ex}

\begin{ex}\upshape
The elliptic curve
$$
E_3 : y^2 = x^3 - 3x^2 + 3x
$$
has additive reduction at 3, and its 3-torsion point $(1,1)$ is of good reduction. 
\end{ex}

\begin{ex}\upshape
The elliptic curve
$$
E_5 : y^2 - 5 y = x^3 + 20x^2 - 15x
$$
has additive reduction at 5, and its 5-torsion point $(1,-1)$ is of good reduction. 
\end{ex}

\begin{ex}\upshape
The elliptic curve
$$
E_7 : y^2 + 7xy  -28y = x^3 + 7x - 35
$$
has additive reduction at 7, and its 7-torsion point $(2,1)$ is of good reduction. 
\end{ex}

\section{Acknowledgements}
It is a pleasure to thank Ronald van Luijk and Sir Peter Swinnerton-Dyer for many useful remarks.

\bibliography{formal1}

\end{document}